\documentclass[12pt]{article}
\usepackage{algorithm}
\usepackage{algpseudocode}
\usepackage{amsfonts}
\usepackage{amsmath}
\usepackage{amssymb} 
\usepackage{amsthm} 
\usepackage[title]{appendix} 
\usepackage{autobreak}
\usepackage{bigdelim} 
\usepackage{bm}
\usepackage{cancel} 
\usepackage{cases}
\usepackage{chngcntr}
\usepackage{colortbl} 
\usepackage{comment}
\usepackage{empheq} 
\usepackage{enumerate}
\usepackage{fancyhdr} 
\usepackage{fancyvrb} 
\usepackage{framed} 
\usepackage{graphicx}
\usepackage[dvipdfmx]{hyperref} 
\hypersetup{
  colorlinks,
  linkcolor={red!50!black},
  citecolor={blue!50!black},
  urlcolor={blue!80!black}
}
\usepackage{mathrsfs} 
\usepackage{multirow} 
\usepackage[sort&compress,numbers]{natbib}
\usepackage{setspace} 
\usepackage{subfig} 
\usepackage[dvipdfmx]{pdflscape} 
\usepackage{pgfplots} 
\usepackage{pxrubrica} 
\usepackage{url}

\usepackage[top=35truemm,bottom=35truemm,left=30truemm,right=30truemm]{geometry} 
\makeatletter

\@addtoreset{equation}{section}
\makeatother

\algtext*{EndWhile}
\algtext*{EndIf}
\algtext*{EndFor}

\theoremstyle{plain}
\newtheorem{theorem}{Theorem}[section]
\newtheorem{corollary}[theorem]{Corollary}
\newtheorem{lemma}[theorem]{Lemma}

\theoremstyle{definition}

\theoremstyle{remark}
\DeclareMathOperator*{\setspan}{span}

\DeclareMathOperator*{\cone}{cone}

\DeclareMathOperator*{\End}{End}

\newcommand{\relmiddle}[1]{\mathrel{}\middle#1\mathrel{}}

\newcommand{\bbF}{\mathbb{F}}

\newcommand{\bbK}{\mathbb{K}}

\newcommand{\bbV}{\mathbb{V}}

\newcommand{\CP}{\mathcal{CP}}
\newcommand{\COP}{\mathcal{COP}}

\newcommand{\RNum}[1]{\uppercase\expandafter{\romannumeral #1\relax}} 
\newcommand{\Rnum}[1]{\lowercase\expandafter{\romannumeral #1\relax}} 

\title{Non-facial exposedness of \\ copositive cones over symmetric cones}

\makeatletter
\let\@fnsymbol\@arabic
\makeatother

\author{
\normalsize
    Mitsuhiro Nishijima\thanks{Center for Advanced Intelligence Project, RIKEN, 1-4-1, Nihonbashi, Chuo-ku, 1030027, Tokyo, Japan. ({\tt mitsuhiro.nishijima@riken.jp}).}
\and
\normalsize
        Bruno F. Louren\c{c}o\thanks{Department of Fundamental Statistical Mathematics, The Institute of Statistical Mathematics, 10-3 Midori-cho, Tachikawa-shi, 1908562, Tokyo, Japan. ({\tt bruno@ism.ac.jp}).}
        }

\begin{document}
\maketitle

\begin{abstract}\noindent
In this paper, we consider copositive cones over symmetric cones and
show that they are never facially exposed when the underlying cone has dimension at least $2$.
We do so by explicitly exhibiting a non-exposed extreme ray.
Our result extends the known fact that the cone of copositive matrices over the nonnegative orthant is not facially exposed in general.
\end{abstract}
\vspace{0.5cm}


\section{Introduction}
Let $\mathbb{K}$ be a closed cone contained in a finite-dimensional real inner product space.
A self-adjoint linear transformation is said to be \emph{copositive over $\mathbb{K}$} if the associated quadratic form is nonnegative over $\mathbb{K}$.
We refer to the cone of self-adjoint linear transformations that are copositive over $\mathbb{K}$ as the \emph{copositive cone over $\mathbb{K}$}.

If $\mathbb{K}$ is the usual nonnegative orthant $\mathbb{R}^n_+$, then the corresponding copositive cone reduces to the cone of copositive matrices in the usual sense~\cite{SB2021}.
We call the copositive cone over the nonnegative orthant the \emph{standard copositive cone}.
Standard copositive cones have been used to reformulate various NP-hard problems as conic linear programs~\cite{BDd+2000,Burer2009}.

Copositive cones over sets beyond nonnegative orthants enable us to convert more NP-hard problems into equivalent conic linear programs~\cite{Burer2012,BD2012,BMP2016}.
Previous studies have investigated the geometry~\cite{GS2013,GST2013}, approximations~\cite{ZVP2006,BD2012,Lasserre2014,NN2024_Approximation,NN2024_Generalizations}, and the membership problem~\cite{Orlitzky2021} for general copositive cones and their duals.

In this paper, we investigate the facial structure, and in particular, the \emph{facial exposedness} of copositive cones over symmetric cones.
Faces and facially exposed convex sets are basic notations in convex analysis~\cite[Section~\mbox{18}]{Rockafellar1970}.
In particular, in the context of optimization, facial exposedness is important for several reasons.
First, it arises when ensuring the invariance of strict complementarity for conic linear programming under duality~\cite[Theorem~1]{CT2008}.
Second, facial exposedness is a necessary condition for a number of useful stronger exposure properties such as \emph{niceness}~\cite[Theorem~3]{Pataki2013_On} (or \emph{facial dual completeness}~\cite{RT2019}), \emph{tangential exposedness}~\cite[Proposition~\mbox{2.2}]{RT2019}, \emph{amenability}~\cite[Proposition~\mbox{13}]{Lourenco2021}, and \emph{projectional exposedness}~\cite{BW1981}, \cite[Corollary~\mbox{4.4}]{ST1990}.
See also \cite{LRS2022} for a discussion on some of those properties.
It is also a necessary condition for certain algebraic properties such as \emph{spectrahedrality}~\cite[Corollary~1]{RG1995} and \emph{hyperbolicity}~\cite[Theorem~\mbox{23}]{Renegar2005}.
See also \cite[Corollary~\mbox{3.5}]{LRS2022} and \cite[Theorem~\mbox{1.1}]{LRS2024} for the connection between these algebraic properties and amenability.
In this way, proving that a cone is \emph{not facially exposed} provides an easy way to certify that it is neither spectrahedral nor hyperbolic.

The facial structure of the standard copositive cone, although not completely understood, has been a subject of several papers.
Findings up to around the year 2021 are covered in the book~\cite{SB2021} and references therein.
See \cite{HA2024,KT2021,KT2022,MST+2024,Nishijima2024,Kostyukova20XX} for more recent results.
Of particular importance is that the standard copositive cone of order $n\ge 2$ has non-exposed extreme rays generated by the matrix $\bm{e}_i\bm{e}_i^\top$ for each $i = 1,\dots,n$~\cite[Theorem~4.4]{Dickinson2011}, where $\bm{e}_i$ is the vector with the $i$th element $1$ and the others $0$.
However, to the best of our knowledge, there is no systematic study on the facial structure of copositive cones over symmetric cones other than the nonnegative orthant.

Let $\mathbb{K}$ be a symmetric cone of dimension at least $2$.
We will show that for every $c$ that generates an extreme ray of $\mathbb{K}$, the corresponding rank-1 tensor $c\otimes c$ generates a non-exposed extreme ray of the copositive cone over $\mathbb{K}$.
In particular, the copositive cone over $\mathbb{K}$ is never facially exposed.
Our result generalizes \cite[Theorem~\mbox{4.4}]{Dickinson2011} to the case of general symmetric cones.
We remark, however, that the proof will not be a straightforward extension of that of \cite[Theorem~\mbox{4.4}]{Dickinson2011}.

The organization of this paper is as follows.
In Section~\ref{sec:Preliminaries}, we introduce and recall some concepts used in this paper, including copositive cones and symmetric cones.
In Section~\ref{sec:main}, we provide a non-exposed extreme ray of copositive cones over symmetric cones.

\section{Preliminaries}\label{sec:Preliminaries}
\subsection{Notation}
Let $\mathbb{R}$ be the set of real numbers.
For an element $x$ in a finite-dimensional real vector space, we define
\begin{align*}
\mathbb{R}x &\coloneqq \{\alpha x \mid \alpha \in \mathbb{R}\},\\
\mathbb{R}_+x &\coloneqq \{\alpha x \mid \alpha \ge 0\}.
\end{align*}
For a subset $S$ of a finite-dimensional real inner product space $V$, we use $\setspan S$ and $S^\perp$ to denote the linear span of $S$ and the space of $x\in V$ such that the inner product between $x$ and $y$ is $0$ for all $y\in S$, respectively.
For a linear mapping $f$ from a finite-dimensional real inner product space to another one, we denote by $f^*$  its adjoint.
For two functions $f$ and $g$, we write $f\circ g$ for the composition of $f$ and $g$.
Note that the product of Jordan algebras introduced in Section~\ref{subsec:symcone_EJA} is also denoted by the symbol $\circ$.

\subsection{Basic properties of linear mappings}
Throughout this subsection, let $(\mathbb{V},\bullet)$ be a finite-dimensional real inner product space with an orthogonal direct sum decomposition
\begin{equation}
\mathbb{V} = \bigoplus_{l=1}^k\mathbb{V}_l. \label{eq:V_direct_sum}
\end{equation}
Let $\End(\mathbb{V})$ be the space of linear transformations on $\mathbb{V}$.
The space $\End(\mathbb{V})$ is equipped with the trace inner product denoted by $\langle \cdot,\cdot\rangle$.
We define $\mathcal{S}(\mathbb{V})$ to be the subspace of $\End(\mathbb{V})$ whose elements are self-adjoint.

For each $i = 1,\dots,k$, let $\mathcal{P}_{\mathbb{V}_i}\colon \mathbb{V} \to \mathbb{V}_i$ be the orthogonal projection onto $\mathbb{V}_i$, i.e.,:
\begin{equation*}
\begin{array}{rccc}
\mathcal{P}_{\mathbb{V}_i}\colon &\mathbb{V} = \bigoplus_{l=1}^k\mathbb{V}_l                &\longrightarrow& \mathbb{V}_i.                   \\
        & \rotatebox{90}{$\in$}&               & \rotatebox{90}{$\in$} \\
        &\sum_{l=1}^kx_l                    & \longmapsto   & x_i
\end{array}
\end{equation*}
Note that
\begin{equation}
\sum_{l=1}^k\mathcal{P}_{\mathbb{V}_l}(x) = x \label{eq:proj_sum}
\end{equation}
holds for all $x \in \mathbb{V}$.
In addition, the adjoint $\mathcal{P}_{\mathbb{V}_l}^*\colon \mathbb{V}_l \to \mathbb{V}$ is the inclusion mapping.
For $\mathcal{A} \in \mathcal{S}(\mathbb{V})$ and for each $i,j = 1,\dots,k$, we define $\mathcal{A}_{i,j} \coloneqq \mathcal{P}_{\mathbb{V}_i} \circ \mathcal{A}|_{\mathbb{V}_j}$, which is a linear mapping from $\mathbb{V}_j$ to $\mathbb{V}_i$.

The following two lemmas justify the notation $\mathcal{A}_{i,j}$ for a self-adjoint linear transformation $\mathcal{A}$.
\begin{lemma}\label{lem:sym_mapping_element}
Let $\mathcal{A} \in \mathcal{S}(\mathbb{V})$.
For each $i,j=1,\dots,k$, $\mathcal{A}_{i,j} = (\mathcal{A}_{j,i})^*$ holds.
\end{lemma}
\begin{proof}
Let $x_i \in \mathbb{V}_i$ and $x_j \in \mathbb{V}_j$ be arbitrary.
To prove this lemma, it is sufficient to show that both $x_i \bullet \mathcal{A}_{i,j}(x_j)$ and $x_i\bullet (\mathcal{A}_{j,i})^*(x_j)$ agree with $x_i \bullet \mathcal{A}(x_j)$.

First, we have
\begin{equation}
x_i \bullet \mathcal{A}_{i,j}(x_j) = x_i \bullet \mathcal{P}_{\mathbb{V}_i}(\mathcal{A}(x_j)) = x_i \bullet \mathcal{A}(x_j), \label{eq:sym_mapping_element_1}
\end{equation}
where we use \eqref{eq:proj_sum} and its orthogonality to derive the second equality.

Second, we have
\begin{equation*}
x_i\bullet (\mathcal{A}_{j,i})^*(x_j) = x_j \bullet \mathcal{A}_{j,i}(x_i) = x_j \bullet \mathcal{A}(x_i) = x_i\bullet \mathcal{A}(x_j),
\end{equation*}
where we use the symmetry of the inner product and the definition of the adjoint $(\mathcal{A}_{j,i})^*$ to derive the first equality, the second equality follows for the same reason as in \eqref{eq:sym_mapping_element_1}, and the third equality holds because of the self-adjointness of $\mathcal{A}$ and the symmetry of the inner product.
This completes the proof.
\end{proof}

\begin{lemma}\label{lem:quad_form}
Let $\mathcal{A} \in \mathcal{S}(\mathbb{V})$.
For each $x \in \mathbb{V}$, we decompose $x$ into $\sum_{l=1}^kx_l$ corresponding to the decomposition~\eqref{eq:V_direct_sum}.
Then it follows that
\begin{equation*}
x\bullet \mathcal{A}(x) = \sum_{i=1}^kx_i\bullet \mathcal{A}_{i,i}(x_i) + 2\sum_{1\le i < j\le k}x_i\bullet \mathcal{A}_{i,j}(x_j).
\end{equation*}
\end{lemma}
\begin{proof}
Since $\mathcal{A}(x) = \sum_{j=1}^k\mathcal{A}|_{\mathbb{V}_j}(x_j)$, we have
\begin{align}
x\bullet \mathcal{A}(x) &= \sum_{i,j=1}^kx_i\bullet \mathcal{A}|_{\mathbb{V}_j}(x_j) \nonumber\\
&= \sum_{i,j=1}^kx_i\bullet (\mathcal{P}_{\mathbb{V}_i}\circ \mathcal{A}|_{\mathbb{V}_j})(x_j)\nonumber\\
&= \sum_{i,j=1}^kx_i\bullet  \mathcal{A}_{i,j}(x_j) \nonumber\\
&= \sum_{i=1}^kx_i\bullet \mathcal{A}_{i,i}(x_i) + \sum_{1\le i < j\le k}x_i\bullet \mathcal{A}_{i,j}(x_j) + \sum_{1 \le j < i\le k}x_i\bullet \mathcal{A}_{i,j}(x_j), \label{eq:quad_form_expansion}
\end{align}
where we use \eqref{eq:proj_sum} and its orthogonality to derive the second equality.
The third term in \eqref{eq:quad_form_expansion} is equal to
\begin{equation}
\sum_{1 \le j < i\le k}x_j\bullet (\mathcal{A}_{i,j})^*(x_i) = \sum_{1 \le j < i\le k}x_j\bullet \mathcal{A}_{j,i}(x_i) = \sum_{1 \le i < j\le k}x_i\bullet \mathcal{A}_{i,j}(x_j), \label{eq:quad_form_expansion_third_term}
\end{equation}
where the first equality follows from Lemma~\ref{lem:sym_mapping_element}.
Since \eqref{eq:quad_form_expansion_third_term} agrees with the second term in \eqref{eq:quad_form_expansion}, we obtain the desired result.
\end{proof}

Lemma~\ref{lem:quad_form} implies that for each $\mathcal{A} \in \mathcal{S}(\mathbb{V})$ and for each $i,j=1,\dots,k$, $\mathcal{A}_{i,j}$ can be regarded as the ``$(i,j)$th element'' of $\mathcal{A}$ when an orthogonal direct sum decomposition~\eqref{eq:V_direct_sum} is fixed.
Therefore, we may write a self-adjoint linear transformation $\mathcal{A}\in\mathcal{S}(\mathbb{V})$ in the following matrix-like form:
\begin{equation}
\begin{pmatrix}
\mathcal{A}_{1,1} & \mathcal{A}_{1,2} & \cdots & \mathcal{A}_{1,k}\\
& \mathcal{A}_{2,2} & \cdots & \mathcal{A}_{2,k}\\
& & \ddots & \vdots\\
& & & \mathcal{A}_{k,k}
\end{pmatrix}, \label{eq:matrix_form}
\end{equation}
where the strictly lower triangular portion of $\mathcal{A}$ can be omitted because of its self-adjointness.
For a subset $\mathcal{S}$ in $\mathcal{S}(\mathbb{V}_k)$, we define
\begin{equation}
\{0\} \oplus \mathcal{S}  \coloneqq \{\mathcal{P}_{\mathbb{V}_k}^* \circ \mathcal{A} \circ \mathcal{P}_{\mathbb{V}_k} \mid \mathcal{A} \in \mathcal{S}\} \subseteq \mathcal{S}(\mathbb{V}). \label{eq:zero_padding}
\end{equation}
Using the matrix-like notation, we can write the set~\eqref{eq:zero_padding} as
\begin{equation*}
\left\{\begin{pmatrix}
0 & \cdots & 0 & 0\\
&\ddots & \vdots & \vdots\\
& & 0 & 0\\
& & & \mathcal{A}
\end{pmatrix} \relmiddle| \mathcal{A}\in\mathcal{S}\right\}.
\end{equation*}

\subsection{Cones and their faces}
Let $(\bbV,\bullet)$ be a finite-dimensional real inner product space.
A set $\bbK \subseteq \bbV$ is called a \emph{cone} if $\alpha x\in \bbK$ for all $\alpha > 0$ and $x\in \bbK$.
For a cone $\bbK$, its \emph{dual cone} is denoted by $\bbK^*$ and is the set of $x\in \bbV$ such that $x\bullet y \ge 0$ for all $y\in \bbK$.

Let $\mathbb{K} \subseteq \bbV$ be a closed cone.
Then
\begin{equation}
\COP(\mathbb{K}) \coloneqq \{\mathcal{A} \in \mathcal{S}(\mathbb{V}) \mid  x\bullet \mathcal{A}(x) \ge 0 \text{ for all $x\in \mathbb{K}$}\} \label{eq:COP_transformation}
\end{equation}
denotes the \emph{copositive cone over $\mathbb{K}$}.
Clearly, the copositive cone $\COP(\mathbb{K})$ is a closed convex cone.
It is known that under the trace inner product induced by $\bullet$, the dual cone of $\eqref{eq:COP_transformation}$ is
\begin{equation}\label{eq:CP}
\CP(\mathbb{K}) = \cone({\left\{a\otimes a \mid a \in \mathbb{K} \right\}}),
\end{equation}
where $\cone({U})$ is the convex cone generated by a subset $U$, see, e.g., \cite[Proposition~1 and Lemma~1]{SZ2003}.
In particular, $\CP(\mathbb{K}) $ is closed. We also recall that for $a,b \in \mathbb{V}$, the tensor product $a\otimes b $  corresponds to
the linear mapping on $\mathbb{V}$ such that $(a\otimes b) (x) = (b \bullet x)a$, for all $x \in \mathbb{V}$.

Next, suppose that $\bbK$ is a closed convex cone.
A nonempty convex subcone $\bbF$ of $\bbK$ is called a \emph{face} of $\bbK$ if for any $a,b\in \bbK$, they belong to $\bbF$ whenever $a + b \in \bbF$.
For a nonzero $x\in \bbK$, if $\mathbb{R}_+x$ is a face of $\bbK$, we say that $x$ \emph{generates} the \emph{extreme ray} $\mathbb{R}_+x$.
A face $\bbF$ of $\bbK$ is said to be \emph{exposed} if there exists $h\in \bbK^*$ such that $\bbF = \bbK \cap \{h\}^\perp$.
If all the faces of $\bbK$ are exposed, then we call $\bbK$ \emph{facially exposed}.

\subsection{Symmetric cone and Euclidean Jordan algebra}\label{subsec:symcone_EJA}
In this subsection, we introduce symmetric cones.
We also present Euclidean Jordan algebras, which are closely related to symmetric cones as mentioned later.
The book written by Faraut and Kor\'{a}nyi~\cite{FK1994} is a standard textbook on this field, but the papers written by Faybusovich~\cite{Faybusovich2008} and Sturm~\cite{Sturm2000_Similarity} are also good references that are more focused on optimization aspects.

A closed cone $\mathbb{K}$ in a finite-dimensional real inner product space $\mathbb{E}$ is called  \emph{symmetric} if it satisfies the following two conditions:\footnote{
In the book written by Faraut and Kor\'{a}nyi~\cite{FK1994}, symmetric cones are \emph{open}.
In the context of optimization, however, symmetric cones are often considered \emph{closed}~\cite{Orlitzky2021,Sturm2000_Similarity} and we also follow the latter.}
\begin{enumerate}[(i)]
\item (Self-duality) $\mathbb{K}^* = \mathbb{K}$,
\item (Homogeneity) For all $x$ and $y$ belonging to the interior of $\mathbb{K}$, there exists a bijective linear transformation $\mathcal{G}\in \End(\mathbb{E})$ such that $\mathcal{G}(\mathbb{K}) = \mathbb{K}$ and $\mathcal{G}(x) = y$.
\end{enumerate}
By the self-duality of $\mathbb{K}$, $\mathbb{K}$ is full-dimensional in $\mathbb{E}$, i.e., $\setspan\mathbb{K} = \mathbb{E}$.

A \emph{Jordan algebra} is a finite-dimensional real vector space $\mathbb{E}$ equipped with a bilinear product $\circ\colon\mathbb{E}\times\mathbb{E}\to \mathbb{E}$
satisfying the following two conditions for all $x,y\in \mathbb{E}$:
\begin{enumerate}[(J1)]
\item (Commutativity) $x \circ y = y \circ x$,
\item (Jordan identity) $x\circ ((x\circ x)\circ y) = (x\circ x)\circ(x\circ y)$.
\end{enumerate}
We assume in this paper that every Jordan algebra $(\mathbb{E},\circ)$ has an identity element, denoted by $e$, concerning the product $\circ$.
A Jordan algebra $(\mathbb{E},\circ)$ is \emph{Euclidean} if there exists an inner product $\bullet\colon\mathbb{E}\times \mathbb{E}\to \mathbb{R}$ satisfying
\begin{enumerate}[(J3)]
\setcounter{enumi}{2}
\item (Associativity) $(x\circ y)\bullet z = x\bullet (y\circ z)$
\end{enumerate}
for all $x,y,z\in \mathbb{E}$.
In this paper, we fix such an associative inner product $\bullet$ and write a Euclidean Jordan algebra as a triple $(\mathbb{E},\circ,\bullet)$.
However, we may merely write $\mathbb{E}$ for $(\mathbb{E},\circ,\bullet)$ if it is clear from the context that $\mathbb{E}$ has the Euclidean Jordan algebraic structure.

Throughout this subsection, let $\mathbb{E}$ be a Euclidean Jordan algebra.
For convenience, let $x^2 \coloneqq x\circ x$ for each $x\in \mathbb{E}$.
It is known that the cone of squares in $\mathbb{E}$ defined as $\mathbb{E}_+ \coloneqq \{x^2\mid x\in\mathbb{E}\}$ is a symmetric cone~\cite[Theorem~\mbox{\RNum{3}.2.1}]{FK1994}.
Conversely, for a given symmetric cone $\mathbb{K}$ in a finite-dimensional real inner product space $(\hat{\mathbb{E}},\hat{\bullet})$, we can define a bilinear product $\hat{\circ}$ on $\hat{\mathbb{E}}$, so that $(\hat{\mathbb{E}},\hat{\circ},\hat{\bullet})$ is a Euclidean Jordan algebra and $\mathbb{K}$ agrees with the cone $\hat{\mathbb{E}}_+$~\cite[Theorem~\mbox{\RNum{3}.3.1}]{FK1994}.

An element $c\in\mathbb{E}$ is termed an \emph{idempotent} if $c^2 = c$.
Note that any idempotent $c\in \mathbb{E}$ belongs to the symmetric cone $\mathbb{E}_+$.
An idempotent $c$ is said to be \emph{primitive} if it is nonzero and cannot be represented as the sum of two nonzero idempotents.
Two idempotents $c$ and $d$ are termed \emph{orthogonal} if $c\circ d = 0$.
The system $c_1,\dots,c_r$ is called a \emph{Jordan frame} if each $c_i$ is a primitive idempotent, they are orthogonal to each other with respect to the product $\circ$, and $\sum_{i=1}^rc_i = e$.
The value $r$ appearing in a Jordan frame is called the \emph{rank} of the Euclidean Jordan algebra and depends only on the algebra~\cite[Section~\mbox{\RNum{3}.1}]{FK1994}.

For an idempotent $c\in \mathbb{E}$ and $\lambda = 0,\frac{1}{2},1$, we define
\begin{equation*}
\mathbb{E}(c,\lambda) \coloneqq \{x \in \mathbb{E} \mid c\circ x = \lambda x\}.
\end{equation*}
The subspaces $\mathbb{E}(c,0)$ and $\mathbb{E}(c,1)$ are Euclidean Jordan subalgebras of $\mathbb{E}$ satisfying $\mathbb{E}(c,0) \circ \mathbb{E}(c,1) = \{0\}$, see ~\cite[Proposition~\mbox{\RNum{4}.1.1}]{FK1994}.
Using the subspaces $\mathbb{E}(c,\lambda)$, the space $\mathbb{E}$ decomposes into the following orthogonal direct sum~\cite[page~\mbox{62}]{FK1994}:
\begin{equation}
\mathbb{E} = \mathbb{E}(c,0)\oplus \mathbb{E}(c,{\textstyle \frac{1}{2}}) \oplus \mathbb{E}(c,1). \label{eq:Peirce_concise}
\end{equation}
We refer to \eqref{eq:Peirce_concise} as the \emph{Peirce decomposition of $\mathbb{E}$ with respect to the idempotent $c$}.
Furthermore, we can decompose $\mathbb{E}$ more finely.
Suppose that the rank of $\mathbb{E}$ is $r$.
We fix a Jordan frame $c_1,\dots,c_r$ of $\mathbb{E}$ and consider the following subspaces of $\mathbb{E}$:
\begin{equation}\label{eq:Eij}
\begin{aligned}
\mathbb{E}_{ii} &\coloneqq \mathbb{E}(c_i,1) = \mathbb{R}c_i& &\text{($i = 1,\dots,r$)},\\
\mathbb{E}_{ij} &\coloneqq \mathbb{E}(c_i,{\textstyle \frac{1}{2}}) \cap \mathbb{E}(c_j,{\textstyle \frac{1}{2}}) & & \text{($i,j = 1,\dots,r$, $i\neq j$)}.
\end{aligned}
\end{equation}
Then we can decompose $\mathbb{E}$ into the following orthogonal direct sum of the above subspaces~\cite[Theorem~\mbox{\RNum{4}.2.1.i}]{FK1994}:
\begin{equation}
\mathbb{E} = \bigoplus_{1\le i\le j\le r}\mathbb{E}_{ij}. \label{eq:Peirce}
\end{equation}
We call \eqref{eq:Peirce} the \emph{Peirce decomposition of $\mathbb{E}$ with respect to the Jordan frame $c_1,\dots,c_r$}.
For each $i,j = 1,\dots,r$ with $i\neq j$ and every $x\in \mathbb{E}_{ij}$, we have
\begin{equation}\label{eq:xij2}
x^2 = \underbrace{c_i\circ x^2}_{ \in \mathbb{E}_{ii}} + \underbrace{c_j\circ x^2}_{ \in \mathbb{E}_{jj}},
\end{equation}
see \cite[Proposition~\mbox{\RNum{4}.1.1}]{FK1994} and its proof.

\section{Main result}\label{sec:main}

\begin{lemma}\label{lem:COP_symcone_face}
{Let $(\mathbb{E},\circ,\bullet)$ be a Euclidean Jordan algebra.}
For an idempotent $c\in \mathbb{E}$, consider the Peirce decomposition~\eqref{eq:Peirce_concise} with respect to $c$.
Then $\{0\}\oplus \COP(\mathbb{E}(c,1)_+)$ is a face of $\COP(\mathbb{E}_+)$.
\end{lemma}

\begin{proof}
{Since $\COP(\mathbb{E}(c,1)_+)$ is a convex cone, so is $\{0\}\oplus \COP(\mathbb{E}(c,1)_+)$.
To prove the inclusion, let $\mathcal{A}\in \{0\}\oplus \COP(\mathbb{E}(c,1)_+)$.
Then there exists $\mathcal{G}\in \COP(\mathbb{E}(c,1)_+)$ such that $\mathcal{A} = \mathcal{P}_{\mathbb{E}(c,1)}^* \circ \mathcal{G} \circ \mathcal{P}_{\mathbb{E}(c,1)}$.
For any $x\in \mathbb{E}_+$, we decompose it into $x = x_0 + x_1$, where $x_0\in \mathbb{E}(c,1)^\perp$ and $x_1\in \mathbb{E}(c,1)$.
It follows from \cite[Proposition~\mbox{32}]{Lourenco2021} that $x_1\in \mathbb{E}(c,1)_+$.
Combining it with $\mathcal{G}\in \COP(\mathbb{E}(c,1)_+)$ yields
\begin{align*}
x\bullet \mathcal{A}(x) &= x\bullet (\mathcal{P}_{\mathbb{E}(c,1)}^* \circ \mathcal{G} \circ \mathcal{P}_{\mathbb{E}(c,1)})(x)\\
&= x_1\bullet \mathcal{G}(x_1)\\
&\ge 0,
\end{align*}
which implies that $\mathcal{A}\in \COP(\mathbb{E}_+)$.
Therefore, $\{0\}\oplus \COP(\mathbb{E}(c,1)_+)$ is a convex subcone of $\COP(\mathbb{E}_+)$.}

Let $r$ be the rank of the Euclidean Jordan algebra $\mathbb{E}$.
We decompose the idempotent $c$ into the sum of orthogonal primitive idempotents $c_p,\dots,c_r$.
We can find $c_1,\dots,c_{p-1}$ such that $c_1,\dots,c_r$ is a Jordan frame of $\mathbb{E}$.
Indeed, since $\mathbb{E}(c,0)$ is a Euclidean Jordan subalgebra of rank $p-1$, we can take a Jordan frame $c_1,\dots,c_{p-1}$ of $\mathbb{E}(c,0)$ by \cite[Theorem~\mbox{\RNum{3}.1.2}]{FK1994}.
Since the identity element of $\mathbb{E}(c,0)$ is $e - c$, the Jordan frame satisfies $e-c = \sum_{i=1}^{p-1}c_i$.
Then $c_1,\dots,c_r$ is a Jordan frame of $\mathbb{E}$.
See also \cite[Lemma~\mbox{23}]{LT2020} for a related result on frame extension.

Consider the Peirce decomposition~\eqref{eq:Peirce} with respect to the Jordan frame $c_1,\dots,c_r$.
We note that $c_p,\dots,c_r$ is a Jordan frame of the algebra $ \mathbb{E}(c,1)$, so we have
\begin{equation}\label{eq:Ec1}
\mathbb{E}(c,1) = \bigoplus_{p\le i\le j\le r}\mathbb{E}_{ij},
\end{equation}
which follows from \eqref{eq:Peirce} applied to $\mathbb{E}(c,1)$.
Let $\preceq$ be the lexicographical order, whence the elements in the set $\{(i,j) \mid 1\le i\le j\le r\}$ are ordered as $(1,1) \preceq (1,2) \preceq \cdots \preceq (1,r) \preceq (2,2) \preceq \cdots \preceq (r,r)$.
For simplicity, when there is no danger of confusion, we write $ij$ for $(i,j)$.
In accordance with the orthogonal decomposition~\eqref{eq:Peirce}, for $\mathcal{A} \in \mathcal{S}(\mathbb{E})$, we use the following matrix-like notation:
\begin{equation*}
\mathcal{A} = (\mathcal{A}_{ij,kl})_{\substack{1\le i \le j\le r \\ 1\le k\le l\le r\\
ij \preceq kl}}.
\end{equation*}
Recall that $\mathcal{A}_{ij,kl} = \mathcal{P}_{\mathbb{E}_{ij}} \circ \mathcal{A}|_{\mathbb{E}_{kl}}$ is the $(ij,kl)$th element of $\mathcal{A}$.

For $\mathcal{A},\mathcal{B}\in \COP(\mathbb{E}_+)$, we suppose that $\mathcal{A} + \mathcal{B} \in \{0\}\oplus \COP(\mathbb{E}(c,1)_+)$.
Then there exists $\mathcal{G}\in \COP(\mathbb{E}(c,1)_+)$ such that
\begin{equation}
\mathcal{A} + \mathcal{B} = \mathcal{P}_{\mathbb{E}(c,1)}^* \circ \mathcal{G} \circ \mathcal{P}_{\mathbb{E}(c,1)}. \label{eq:COP_symcone_face_asm}
\end{equation}

We see that $\mathcal{P}_{\mathbb{E}(c,1)}\circ \mathcal{A}|_{\mathbb{E}(c,1)}$  belongs to $\COP(\mathbb{E}(c,1)_+)$.
Indeed, let $x\in \mathbb{E}(c,1)_+$ be arbitrary.
Then we have
\begin{equation*}
0 \overset{\scriptsize \text{(a)}}\le x \bullet \mathcal{A}(x) \overset{\scriptsize \text{(b)}}= x \bullet \mathcal{A}|_{\mathbb{E}(c,1)}(x) \overset{\scriptsize \text{(c)}}= x \bullet (\mathcal{P}_{\mathbb{E}(c,1)}\circ \mathcal{A}|_{\mathbb{E}(c,1)})(x),
\end{equation*}
where (a) holds because $x \in \mathbb{E}(c,1)_+ \subseteq \mathbb{E}_+$ and $\mathcal{A} \in \COP(\mathbb{E}_+)$, (b) follows from $x \in \mathbb{E}(c,1)$, and (c) is a consequence of the orthogonality of \eqref{eq:Peirce_concise}.
In addition, $\mathcal{P}_{\mathbb{E}(c,1)}\circ \mathcal{A}|_{\mathbb{E}(c,1)}$ is the ``principal submatrix'' of $\mathcal{A}$ obtained by extracting the $\frac{1}{2}(r-p+1)(r-p+2)$ rows and columns indexed by $(p,p),(p,p+1)\dots,(r,r)$.
The same is true for $\mathcal{B}$.

Thus, to prove $\mathcal{A},\mathcal{B}\in \{0\}\oplus \COP(\mathbb{E}(c,1)_+)$, it is sufficient to show that
\begin{equation}
\mathcal{A}_{ij,kl} = \mathcal{B}_{ij,kl} = 0 \label{eq:to_prove}
\end{equation}
for all $(i,j,k,l)$ satisfying $1\le i\le p-1$, $1 \le i \le j\le r$, $1 \le k \le l \le r$, and $ij \preceq kl$.
See Figure~\ref{fig:induction_proof} for an illustration of this proof.
\begin{figure}
\begin{center}
\includegraphics{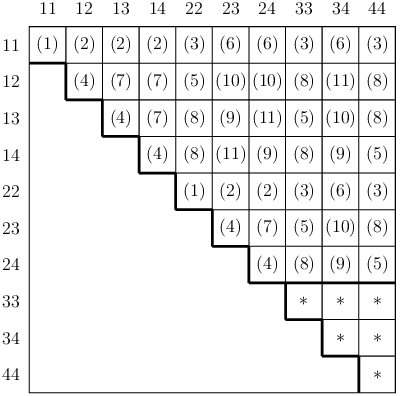}
\end{center}
\caption{Illustration of the proof of Lemma~\ref{lem:COP_symcone_face} in the case of $r =  4$ and $p = 3$.
The cells marked $(i)$ correspond to Case~$i$ for each $i = 1,\dots,11$.
The $3\times 3$ lower-right block marked the symbol $*$ corresponds to $\mathcal{P}_{\mathbb{E}(c,1)}\circ \mathcal{A}|_{\mathbb{E}(c,1)}$ and $\mathcal{P}_{\mathbb{E}(c,1)}\circ \mathcal{B}|_{\mathbb{E}(c,1)}$.
}
\label{fig:induction_proof}
\end{figure}
We also recall that $\mathcal{A}_{ij,kl}$ is a linear mapping between $\mathbb{E}_{kl}$ and
$\mathbb{E}_{ij}$, so we have the following characterization:
\begin{equation}\label{eq:zero_condition}
 \mathcal{A}_{ij,kl} = 0 \Longleftrightarrow x_{ij}\bullet\mathcal{A}_{ij,kl}(x_{kl}) = 0 \text{ {for all $x_{ij} \in \mathbb{E}_{ij}$ and $x_{kl} \in \mathbb{E}_{kl}$}}.
\end{equation}
Moving on, a {quadruple} $(i,j,k,l)$ satisfying {the following three conditions}
\begin{equation}
i \le j,\ k \le l,\ ij \preceq kl \label{eq:three_conds}
\end{equation}
must fall into exactly one of the following eleven cases:
\begin{enumerate}[C{a}se~1:]
\setlength{\leftskip}{1.5em}
\item $i = j = k = l$; \label{enum:case1}
\item $i = j = k \neq l$;\label{enum:case2}
\item $i = j \neq k = l$;\label{enum:case3}
\item $i = k \neq j = l$;\label{enum:case4}
\item $i \neq j = k = l$;\label{enum:case5}
\item $i = j$ and $i,k,l$ are different from each other;\label{enum:case6}
\item $i = k$ and $i,j,l$ are different from each other;\label{enum:case7}
\item $k = l$ and $i,j,l$ are different from each other;\label{enum:case8}
\item $j = l$ and $i,k,l$ are different from each other;\label{enum:case9}
\item $j = k$ and $i,j,l$ are different from each other;\label{enum:case10}
\item $i,j,k,l$ are different from each other.\label{enum:case11}
\end{enumerate}
See \ref{apdx:eleven_cases} for the reasoning behind this case division.
In the following discussion, for each $x \in \mathbb{E}$, we write $\alpha(x) \coloneqq x\bullet \mathcal{A}(x)$ and $\beta(x) \coloneqq x\bullet \mathcal{B}(x)$.
Similarly, we define
\begin{equation}
\gamma(x) \coloneqq x \bullet (\mathcal{P}_{\mathbb{E}(c,1)}^* \circ \mathcal{G} \circ \mathcal{P}_{\mathbb{E}(c,1)})(x). \label{eq:COP_symcone_face_gamma}
\end{equation}
Note that by \eqref{eq:COP_symcone_face_asm}, we have
\begin{equation}
\alpha(x) + \beta(x) = \gamma(x) \label{eq:COP_symcone_face_abg}
\end{equation}
for all $x\in \mathbb{E}$.
Also note that $\alpha(x)$ and $\beta(x)$ are nonnegative for any $x\in \mathbb{E}_+$ since $\mathcal{A},\mathcal{B} \in \COP(\mathbb{E}_+)$.

In what follows, in order to show that \eqref{eq:to_prove} holds for any $(i,j,k,l)$ satisfying each of the eleven cases above, we adopt the following convention.
Whenever we represent a linear transformation in the matrix-like form as in \eqref{eq:matrix_form}, if the symbol $(i)$ appears in a matrix entry, then it means
that entry is $0$ because of Case~$i$.
For example, in \eqref{eq:COP_symcone_face_case2_matrixlike}, ``\eqref{enum:case1}'' indicates that the corresponding entry is $0$ because of Case~\ref{enum:case1}.
In \eqref{eq:COP_symcone_face_case5_matrixlike}, four entries are $0$ because of Cases~\ref{enum:case1} to~\ref{enum:case4}.

\fbox{Case~\ref{enum:case1}}
On the one hand, since $c_i \in \mathbb{E}_+$, we have $\alpha(c_i),\beta(c_i) \ge 0$.
On the other hand, it follows from $c_i \not\in \mathbb{E}(c,1)$ and \eqref{eq:COP_symcone_face_gamma} that $\gamma(c_i) = 0$.
Therefore, by \eqref{eq:COP_symcone_face_abg}, both $\alpha(c_i) = c_i \bullet \mathcal{A}_{ii,ii}(c_i)$ and $\beta(c_i) = c_i \bullet \mathcal{B}_{ii,ii}(c_i)$ must be $0$.
Since $\mathbb{E}_{ii} = \mathbb{R}c_i$ (see \eqref{eq:Eij}), we have
\begin{equation*}
(a c_i) \bullet \mathcal{A}_{ii,ii}(bc_i)= ab (c_i \bullet \mathcal{A}_{ii,ii}(c_i)) = 0
\end{equation*}
for all $a,b\in\mathbb{R}$.
In view of \eqref{eq:zero_condition} and since an analogous argument holds for $\mathcal{B}$, we obtain $\mathcal{A}_{ii,ii} = \mathcal{B}_{ii,ii} = 0$.

\fbox{Case~\ref{enum:case2}}
For any $x_{il} \in \mathbb{E}_{il}$ and $\epsilon > 0$, let
\begin{equation*}
x(\epsilon) \coloneqq (c_i + \epsilon x_{il})^2 = \underbrace{c_i + \epsilon^2c_i\circ x_{il}^2}_{\in \mathbb{E}_{ii}} + \underbrace{\epsilon x_{il}}_{\in \mathbb{E}_{il}} + \underbrace{\epsilon^2c_l\circ x_{il}^2}_{\in \mathbb{E}_{ll}} \in \mathbb{E}_+,
\end{equation*}
where the last equality holds by \eqref{eq:Eij} and \eqref{eq:xij2}.
Then on the one hand, using Lemma~\ref{lem:quad_form} and considering the orthogonality between the $\mathbb{E}_{ij}$ we have
\begin{align}
0 &\le \alpha(x(\epsilon))\nonumber\\
& = x(\epsilon)\bullet\!\!\!\bordermatrix{
& ii & il & ll \cr
& \eqref{enum:case1} & \mathcal{A}_{ii,il} & \mathcal{A}_{ii,ll} \cr
&  & \mathcal{A}_{il,il} & \mathcal{A}_{il,ll} \cr
&  &  & \mathcal{A}_{ll,ll}}(x(\epsilon)) \label{eq:COP_symcone_face_case2_matrixlike}\\
&= 2\epsilon c_i\bullet \mathcal{A}_{ii,il}(x_{il}) + \epsilon^2\{x_{il}\bullet \mathcal{A}_{il,il}(x_{il}) + 2c_i\bullet \mathcal{A}_{ii,ll}(c_l\circ x_{il}^2)\}\nonumber\\
&\quad +2\epsilon^3\{(c_i\circ x_{il}^2)\bullet\mathcal{A}_{ii,il}(x_{il}) + x_{il}\bullet \mathcal{A}_{il,ll}(c_l\circ x_{il}^2)\} + O(\epsilon^4).\label{eq:COP_symcone_face_case2_alpha}
\end{align}
The matrix-like notation in \eqref{eq:COP_symcone_face_case2_matrixlike} is for the case $i < l$, but the calculation is also valid if $l < i$.
By replacing $\mathcal{A}$ in \eqref{eq:COP_symcone_face_case2_alpha} with $\mathcal{B}$, we can also calculate $\beta(x(\epsilon))$, which is greater that or equal to $0$.
On the other hand, recalling \eqref{eq:Peirce_concise}, \eqref{eq:Ec1}, and the fact that $i \leq p-1$ holds by assumption, we have
\begin{equation*}
\mathcal{P}_{\mathbb{E}(c,1)}(x(\epsilon)) = \begin{cases} \epsilon^2c_l\circ x_{il}^2 & (\text{if $p\le l\le r$}),\\
0 & (\text{otherwise}),
\end{cases}
\end{equation*}
which implies that $\gamma(x(\epsilon)) = O(\epsilon^4)$.
Therefore, setting $x = x(\epsilon)$ in  \eqref{eq:COP_symcone_face_abg}, dividing by $2\epsilon$, and letting $\epsilon \downarrow 0$, we obtain
\begin{equation*}
c_i\bullet \mathcal{A}_{ii,il}(x_{il}) + c_i\bullet \mathcal{B}_{ii,il}(x_{il}) = 0.
\end{equation*}
Since $\alpha(x(\epsilon))$ is nonnegative for all $\epsilon > 0$,
\begin{equation*}
c_i\bullet \mathcal{A}_{ii,il}(x_{il}) = \lim_{\epsilon \downarrow 0}\frac{\alpha(x(\epsilon))}{2\epsilon}
\end{equation*}
is also nonnegative.
Similarly, $c_i\bullet \mathcal{B}_{ii,il}(x_{il})$ is nonnegative.
Thus, we have
\begin{equation*}
c_i\bullet \mathcal{A}_{ii,il}(x_{il}) = c_i\bullet \mathcal{B}_{ii,il}(x_{il}) = 0.
\end{equation*}
Since $\mathbb{E}_{ii} = \mathbb{R}c_i$ and $x_{il} \in \mathbb{E}_{il}$ is arbitrary, we obtain $\mathcal{A}_{ii,il} = \mathcal{B}_{ii,il} = 0$, again
by \eqref{eq:zero_condition}.

\fbox{Case~\ref{enum:case3}}
For any $\epsilon > 0$, let $x(\epsilon) \coloneqq c_i + \epsilon c_l \in\mathbb{E}_+$.
Then on the one hand, from Case~\ref{enum:case1}, we have $\mathcal{A}_{ii,ii} = 0$ so
\begin{equation*}
0 \le \alpha(x(\epsilon)) = 2\epsilon c_i\bullet \mathcal{A}_{ii,ll}(c_l) + \epsilon^2c_l\bullet \mathcal{A}_{ll,ll}(c_l).
\end{equation*}
On the other hand, we have
\begin{equation*}
\mathcal{P}_{\mathbb{E}(c,1)}(x(\epsilon)) = \begin{cases} \epsilon c_l & (\text{if $p\le l\le r$}),\\
0 & (\text{otherwise}),
\end{cases}
\end{equation*}
which implies that $\gamma(x(\epsilon)) = O(\epsilon^2)$.
Therefore, setting $x = x(\epsilon)$ in \eqref{eq:COP_symcone_face_abg}, dividing by $2\epsilon$, and letting $\epsilon \downarrow 0$, we obtain
\begin{equation*}
c_i\bullet \mathcal{A}_{ii,ll}(c_l) = c_i\bullet \mathcal{B}_{ii,ll}(c_l) = 0.
\end{equation*}
Recalling that $\mathbb{E}_{ii} = \mathbb{R}c_i$ and $\mathbb{E}_{ll} = \mathbb{R}c_l$ hold, we obtain $\mathcal{A}_{ii,ll} = \mathcal{B}_{ii,ll} = 0$ from \eqref{eq:zero_condition}.

\fbox{Case~\ref{enum:case4}}
For any $x_{il} \in \mathbb{E}_{il}$ and $\epsilon > 0$, let $x(\epsilon)$ be the same as in Case~\ref{enum:case2}.
Then on the one hand, it follows from Cases~\ref{enum:case1}, \ref{enum:case2}, \ref{enum:case3}, and a computation analogous to \eqref{eq:COP_symcone_face_case2_alpha} that
\begin{equation*}
0 \le \alpha(x(\epsilon)) = \epsilon^2x_{il}\bullet \mathcal{A}_{il,il}(x_{il}) + O(\epsilon^3).
\end{equation*}
On the other hand, we have $\gamma(x(\epsilon)) = O(\epsilon^4)$.
Therefore, setting $x = x(\epsilon)$ in \eqref{eq:COP_symcone_face_abg}, dividing by $\epsilon^2$, and letting $\epsilon \downarrow 0$, we obtain
\begin{equation*}
x_{il}\bullet \mathcal{A}_{il,il}(x_{il}) = x_{il}\bullet \mathcal{B}_{il,il}(x_{il}) = 0.
\end{equation*}
Since $x_{il} \in \mathbb{E}_{il}$ is arbitrary, we obtain $\mathcal{A}_{il,il} = \mathcal{B}_{il,il} = 0$.

\fbox{Case~\ref{enum:case5}}
For any $x_{il} \in \mathbb{E}_{il}$ and $\epsilon > 0$, let
\begin{equation*}
x(\epsilon) \coloneqq (c_i + \epsilon^2x_{il})^2 + \epsilon^3c_l = \underbrace{c_i + \epsilon^4c_i\circ x_{il}^2}_{\in \mathbb{E}_{ii}} + \underbrace{\epsilon^2x_{il}}_{\in \mathbb{E}_{il}} + \underbrace{\epsilon^3c_l + \epsilon^4c_l\circ x_{il}^2}_{\in \mathbb{E}_{ll}} \in \mathbb{E}_+.
\end{equation*}
Then on the one hand, we have
\begin{align}
0 &\le \alpha(x(\epsilon)) \nonumber\\
&= x(\epsilon)\bullet\!\!\!\bordermatrix{
& ii & il & ll \cr
& \eqref{enum:case1} & \eqref{enum:case2} & \eqref{enum:case3} \cr
&  & \eqref{enum:case4} & \mathcal{A}_{il,ll} \cr
&  &  & \mathcal{A}_{ll,ll}}(x(\epsilon)) \label{eq:COP_symcone_face_case5_matrixlike}\\
&= 2\epsilon^5x_{il}\bullet \mathcal{A}_{il,ll}(c_l) + O(\epsilon^6).\nonumber
\end{align}
On the other hand, we have $\gamma(x(\epsilon)) = O(\epsilon^6)$.
Therefore, setting $x = x(\epsilon)$ in \eqref{eq:COP_symcone_face_abg}, dividing by $2\epsilon^5$, and letting $\epsilon \downarrow 0$, we obtain
\begin{equation*}
x_{il}\bullet \mathcal{A}_{il,ll}(c_l) = x_{il}\bullet \mathcal{B}_{il,ll}(c_l) = 0.
\end{equation*}
Since $x_{il} \in \mathbb{E}_{il}$ is arbitrary, we obtain $\mathcal{A}_{il,ll} = \mathcal{B}_{il,ll} = 0$.

\fbox{Case~\ref{enum:case6}}
For any $x_{kl} \in \mathbb{E}_{kl}$ and $\epsilon > 0$, let
\begin{align*}
x(\epsilon) \coloneqq c_i + \epsilon(c_k + x_{kl})^2 = \underbrace{c_i}_{\in \mathbb{E}_{ii}} + \underbrace{\epsilon c_k + \epsilon c_k\circ x_{kl}^2}_{\in \mathbb{E}_{kk}} + \underbrace{\epsilon x_{kl}}_{\in \mathbb{E}_{kl}} + \underbrace{\epsilon c_l\circ x_{kl}^2}_{\in \mathbb{E}_{ll}} \in \mathbb{E}_+.
\end{align*}
Then on the one hand, we have
\begin{align*}
0 &\le \alpha(x(\epsilon))\nonumber\\
&= x(\epsilon)\bullet\!\!\!\bordermatrix{
& ii & kk & kl & ll \cr
& \eqref{enum:case1} & \eqref{enum:case3} & \mathcal{A}_{ii,kl} & \eqref{enum:case3} \cr
&  & \mathcal{A}_{kk,kk} & \mathcal{A}_{kk,kl} & \mathcal{A}_{kk,ll}\cr
&  &  & \mathcal{A}_{kl,kl} & \mathcal{A}_{kl,ll} \cr
&  &  &  & \mathcal{A}_{ll,ll}
}(x(\epsilon)) \label{eq:COP_symcone_face_case6_matrixlike}\\
&= 2\epsilon c_i\bullet \mathcal{A}_{ii,kl}(x_{kl}) + O(\epsilon^2).\nonumber
\end{align*}
On the other hand, we have $\gamma(x(\epsilon)) = O(\epsilon^2)$.
Therefore, setting $x = x(\epsilon)$ in \eqref{eq:COP_symcone_face_abg}, dividing by $2\epsilon$, and letting $\epsilon \downarrow 0$, we obtain
\begin{equation*}
c_i\bullet \mathcal{A}_{ii,kl}(x_{kl}) = c_i\bullet \mathcal{B}_{ii,kl}(x_{kl}) = 0.
\end{equation*}
Since $x_{kl} \in \mathbb{E}_{kl}$ is arbitrary, we obtain $\mathcal{A}_{ii,kl} = \mathcal{B}_{ii,kl} = 0$.

\fbox{Case~\ref{enum:case7}}
For any $x_{ij}\in\mathbb{E}_{ij}$, $x_{il}\in\mathbb{E}_{il}$, and $\epsilon > 0$, let
\begin{align*}
x(\epsilon) &\coloneqq (c_i + \epsilon x_{ij})^2 + (c_i + \epsilon x_{il})^2\\
&= \underbrace{2c_i + \epsilon^2c_i\circ x_{ij}^2 + \epsilon^2 c_i\circ x_{il}^2}_{\in \mathbb{E}_{ii}} + \underbrace{\epsilon x_{ij}}_{\in \mathbb{E}_{ij}} + \underbrace{\epsilon x_{il}}_{\in \mathbb{E}_{il}}  + \underbrace{\epsilon^2 c_j\circ x_{ij}^2}_{\in \mathbb{E}_{jj}} + \underbrace{\epsilon^2 c_l\circ x_{il}^2}_{\in \mathbb{E}_{ll}} \in \mathbb{E}_+.
\end{align*}
Then on the one hand, we have
\begin{align*}
0 &\le \alpha(x(\epsilon))\nonumber\\
& = x(\epsilon)\bullet\!\!\!\bordermatrix{
& ii & ij & il & jj & ll \cr
& \eqref{enum:case1} & \eqref{enum:case2} & \eqref{enum:case2} & \eqref{enum:case3} & \eqref{enum:case3}\cr
&  & \eqref{enum:case4} & \mathcal{A}_{ij,il} & \eqref{enum:case5} & \mathcal{A}_{ij,ll}\cr
&  &  & \eqref{enum:case4} & \mathcal{A}_{il,jj} & \eqref{enum:case5} \cr
&  &  &  & \mathcal{A}_{jj,jj} & \mathcal{A}_{jj,ll} \cr
&  &  &  &  & \mathcal{A}_{ll,ll}
}(x(\epsilon)) \label{eq:COP_symcone_face_case7_matrixlike}\\
&= 2\epsilon^2 x_{ij}\bullet \mathcal{A}_{ij,il}(x_{il}) + O(\epsilon^3).\nonumber
\end{align*}
On the other hand, we have $\gamma(x(\epsilon)) = O(\epsilon^4)$.
Therefore, setting $x = x(\epsilon)$ in \eqref{eq:COP_symcone_face_abg}, dividing by $2\epsilon^2$, and letting $\epsilon \downarrow 0$, we obtain
\begin{equation*}
x_{ij}\bullet \mathcal{A}_{ij,il}(x_{il}) = x_{ij}\bullet \mathcal{B}_{ij,il}(x_{il}) = 0.
\end{equation*}
Since $x_{ij} \in \mathbb{E}_{ij}$ and $x_{il} \in \mathbb{E}_{il}$ are arbitrary, we obtain $\mathcal{A}_{ij,il} = \mathcal{B}_{ij,il} = 0$.

\fbox{Case~\ref{enum:case8}}
For any $x_{ij}\in\mathbb{E}_{ij}$ and $\epsilon > 0$, let
\begin{align*}
x(\epsilon) \coloneqq (c_i + \epsilon x_{ij})^2 + \epsilon^2c_l = \underbrace{c_i + \epsilon^2c_i\circ x_{ij}^2}_{\in \mathbb{E}_{ii}} + \underbrace{\epsilon x_{ij}}_{\in \mathbb{E}_{ij}} + \underbrace{\epsilon^2 c_j\circ x_{ij}^2}_{\in \mathbb{E}_{jj}} + \underbrace{\epsilon^2 c_l}_{\in \mathbb{E}_{ll}} \in \mathbb{E}_+.
\end{align*}
Then on the one hand, we have
\begin{align*}
0 &\le \alpha(x(\epsilon))\nonumber\\
&= x(\epsilon)\bullet\!\!\!\bordermatrix{
& ii & ij & jj & ll \cr
& \eqref{enum:case1} & \eqref{enum:case2} & \eqref{enum:case3} & \eqref{enum:case3} \cr
&  & \eqref{enum:case4} & \eqref{enum:case5} & \mathcal{A}_{ij,ll}\cr
&  &  & \mathcal{A}_{jj,jj} & \mathcal{A}_{jj,ll} \cr
&  &  &  & \mathcal{A}_{ll,ll}
}(x(\epsilon)) \label{eq:COP_symcone_face_case8_matrixlike}\\
&= 2\epsilon^3 x_{ij}\bullet \mathcal{A}_{ij,ll}(c_l) + O(\epsilon^4).\nonumber
\end{align*}
On the other hand, we have $\gamma(x(\epsilon)) = O(\epsilon^4)$.
Therefore, setting $x = x(\epsilon)$ in \eqref{eq:COP_symcone_face_abg}, dividing by $2\epsilon^3$, and letting $\epsilon \downarrow 0$, we obtain
\begin{equation*}
x_{ij}\bullet \mathcal{A}_{ij,ll}(c_l) = x_{ij}\bullet \mathcal{B}_{ij,ll}(c_l) = 0.
\end{equation*}
Since $x_{ij} \in \mathbb{E}_{ij}$ is arbitrary, we obtain $\mathcal{A}_{ij,ll} = \mathcal{B}_{ij,ll} = 0$.

\fbox{Case~\ref{enum:case9}}
For any $x_{il}\in\mathbb{E}_{il}$, $x_{kl} \in \mathbb{E}_{kl}$, and $\epsilon > 0$, let
\begin{align*}
x(\epsilon) &\coloneqq (c_i + \epsilon x_{il})^2 + \epsilon^2(c_k + x_{kl})^2\\
&= \underbrace{c_i + \epsilon^2c_i\circ x_{il}^2}_{\in \mathbb{E}_{ii}} + \underbrace{\epsilon x_{il}}_{\in \mathbb{E}_{il}} \\
&\quad + \underbrace{\epsilon^2c_k + \epsilon^2c_k\circ x_{kl}^2}_{\in \mathbb{E}_{kk}} + \underbrace{\epsilon^2 x_{kl}}_{\in \mathbb{E}_{kl}}
+ \underbrace{\epsilon^2c_l\circ x_{kl}^2 + \epsilon^2 c_l\circ x_{il}^2}_{\in \mathbb{E}_{ll}} \in \mathbb{E}_+.
\end{align*}
Then on the one hand, we have
\begin{align*}
0 &\le \alpha(x(\epsilon))\nonumber\\
& = x(\epsilon)\bullet\!\!\!\bordermatrix{
& ii & il & kk & kl & ll \cr
& \eqref{enum:case1} & \eqref{enum:case2} & \eqref{enum:case3} & \eqref{enum:case6} & \eqref{enum:case3}\cr
&  & \eqref{enum:case4} & \eqref{enum:case8} & \mathcal{A}_{il,kl} & \eqref{enum:case5}\cr
&  &  & \mathcal{A}_{kk,kk} & \mathcal{A}_{kk,kl} & \mathcal{A}_{kk,ll} \cr
&  &  &  & \mathcal{A}_{kl,kl} & \mathcal{A}_{kl,ll} \cr
&  &  &  &  & \mathcal{A}_{ll,ll}
}(x(\epsilon)) \label{eq:COP_symcone_face_case9_matrixlike}\\
&= 2\epsilon^3 x_{il}\bullet \mathcal{A}_{il,kl}(x_{kl}) + O(\epsilon^4).\nonumber
\end{align*}
On the other hand, we have $\gamma(x(\epsilon)) = O(\epsilon^4)$.
Therefore, setting $x = x(\epsilon)$ in \eqref{eq:COP_symcone_face_abg}, dividing by $2\epsilon^3$, and letting $\epsilon \downarrow 0$, we obtain
\begin{equation*}
x_{il}\bullet \mathcal{A}_{il,kl}(x_{kl}) = x_{il}\bullet \mathcal{B}_{il,kl}(x_{kl}) = 0.
\end{equation*}
Since $x_{il} \in \mathbb{E}_{il}$ and $x_{kl} \in \mathbb{E}_{kl}$ are arbitrary, we obtain $\mathcal{A}_{il,kl} = \mathcal{B}_{il,kl} = 0$.

\fbox{Case~\ref{enum:case10}}
For any $x_{ij}\in\mathbb{E}_{ij}$, $x_{jl}\in\mathbb{E}_{jl}$, and $\epsilon > 0$, let
\begin{align*}
x(\epsilon) &\coloneqq (c_i + \epsilon x_{ij})^2 + \epsilon^2(c_j + x_{jl})^2\\
&= \underbrace{c_i + \epsilon^2c_i\circ x_{ij}^2}_{\in \mathbb{E}_{ii}} + \underbrace{\epsilon x_{ij}}_{\in \mathbb{E}_{ij}} \\
&\quad + \underbrace{\epsilon^2c_j + \epsilon^2c_j\circ x_{ij}^2 + \epsilon^2c_j\circ x_{jl}^2}_{\in \mathbb{E}_{jj}} + \underbrace{\epsilon^2x_{jl}}_{\in \mathbb{E}_{jl}} + \underbrace{\epsilon^2c_l\circ x_{jl}^2}_{\in \mathbb{E}_{ll}} \in \mathbb{E}_+.
\end{align*}
Then on the one hand, we have
\begin{align*}
0 &\le \alpha(x(\epsilon))\nonumber\\
& = x(\epsilon)\bullet\!\!\!\bordermatrix{
& ii & ij & jj & jl & ll \cr
& \eqref{enum:case1} & \eqref{enum:case2} & \eqref{enum:case3} & \eqref{enum:case6} & \eqref{enum:case3}\cr
&  & \eqref{enum:case4} & \eqref{enum:case5} & \mathcal{A}_{ij,jl} & \eqref{enum:case8}\cr
&  &  & \mathcal{A}_{jj,jj} & \mathcal{A}_{jj,jl} & \mathcal{A}_{jj,ll} \cr
&  &  &  & \mathcal{A}_{jl,jl} & \mathcal{A}_{jl,ll} \cr
&  &  &  &  & \mathcal{A}_{ll,ll}
}(x(\epsilon)) \label{eq:COP_symcone_face_case10_matrixlike}\\
&= 2\epsilon^3 x_{ij}\bullet \mathcal{A}_{ij,jl}(x_{jl}) + O(\epsilon^4).\nonumber
\end{align*}
On the other hand, we have $\gamma(x(\epsilon)) = O(\epsilon^4)$.
Therefore, setting $x = x(\epsilon)$ in \eqref{eq:COP_symcone_face_abg}, dividing by $2\epsilon^3$, and letting $\epsilon \downarrow 0$, we obtain
\begin{equation*}
x_{ij}\bullet \mathcal{A}_{ij,jl}(x_{jl}) = x_{ij}\bullet \mathcal{B}_{ij,jl}(x_{jl}) = 0.
\end{equation*}
Since $x_{ij} \in \mathbb{E}_{ij}$ and $x_{jl} \in \mathbb{E}_{jl}$ are arbitrary, we obtain $\mathcal{A}_{ij,jl} = \mathcal{B}_{ij,jl} = 0$.

\fbox{Case~\ref{enum:case11}}
For any $x_{ij}\in\mathbb{E}_{ij}$, $x_{kl}\in\mathbb{E}_{kl}$, and $\epsilon > 0$, let
\begin{align*}
x(\epsilon) &\coloneqq (c_i + \epsilon x_{ij})^2 + \epsilon^2(c_k + x_{kl})^2\\
&= \underbrace{c_i + \epsilon^2 c_i\circ x_{ij}^2}_{\in \mathbb{E}_{ii}} + \underbrace{\epsilon x_{ij}}_{\in \mathbb{E}_{ij}} + \underbrace{\epsilon^2 c_j\circ x_{ij}^2}_{\mathbb{E}_{jj}}\\
&\quad + \underbrace{\epsilon^2 c_k + \epsilon^2 c_k\circ x_{kl}^2}_{\in \mathbb{E}_{kk}} + \underbrace{\epsilon^2 x_{kl}}_{\in \mathbb{E}_{kl}} + \underbrace{\epsilon^2 c_l\circ x_{kl}^2}_{\in \mathbb{E}_{ll}} \in \mathbb{E}_+.
\end{align*}
Then on the one hand, we have
\begin{align*}
0 &\le \alpha(x(\epsilon)) \nonumber \\
& = x(\epsilon)\bullet\!\!\!\bordermatrix{
& ii & ij & jj & kk & kl & ll \cr
& \eqref{enum:case1} & \eqref{enum:case2} & \eqref{enum:case3} & \eqref{enum:case3} &\eqref{enum:case6} & \eqref{enum:case3}\cr
&  & \eqref{enum:case4} & \eqref{enum:case5} & \eqref{enum:case8} & \mathcal{A}_{ij,kl} & \eqref{enum:case8}\cr
&  &  & \mathcal{A}_{jj,jj} & \mathcal{A}_{jj,kk} & \mathcal{A}_{jj,kl} & \mathcal{A}_{jj,ll}\cr
&  &  &  & \mathcal{A}_{kk,kk} & \mathcal{A}_{kk,kl} & \mathcal{A}_{kk,ll} \cr
&  &  &  &  & \mathcal{A}_{kl,kl} & \mathcal{A}_{kl,ll} \cr
&  &  &  &  &  & \mathcal{A}_{ll,ll} \cr
}(x(\epsilon)) \label{eq:COP_symcone_face_case11_matrixlike}\\
&= 2\epsilon^3 x_{ij}\bullet \mathcal{A}_{ij,kl}(x_{kl}) + O(\epsilon^4).\nonumber
\end{align*}
On the other hand, we have $\gamma(x(\epsilon)) = O(\epsilon^4)$.
Therefore, setting $x = x(\epsilon)$ in \eqref{eq:COP_symcone_face_abg}, dividing by $2\epsilon^3$, and letting $\epsilon \downarrow 0$, we obtain
\begin{equation*}
x_{ij}\bullet \mathcal{A}_{ij,kl}(x_{kl}) = x_{ij}\bullet \mathcal{B}_{ij,kl}(x_{kl}) = 0.
\end{equation*}
Since $x_{ij} \in \mathbb{E}_{ij}$ and $x_{kl} \in \mathbb{E}_{kl}$ are arbitrary, we obtain $\mathcal{A}_{ij,kl} = \mathcal{B}_{ij,kl} = 0$.
\end{proof}

\begin{corollary}\label{cor:COP_symcone_face_geometric}
Let $\mathbb{K}$ be a symmetric cone in a finite-dimensional real inner product space $\mathbb{E}$, and let $\mathbb{F}$ be a face of $\mathbb{K}$.
Also, let $\mathbb{V}_1 \coloneqq  (\setspan\mathbb{F})^\perp$, $\mathbb{V}_2 \coloneqq \setspan\mathbb{F}$, so that $\mathbb{E} = \mathbb{V}_1 \oplus \mathbb{V}_2$.
Then following the definition in \eqref{eq:zero_padding}, $\{0\} \oplus \COP(\mathbb{F})$ is a face of $\COP(\mathbb{K})$.
\end{corollary}

\begin{proof}
Let  $\circ$ be a bilinear product on $\mathbb{E}$ such that $(\mathbb{E},\circ,\bullet)$ is a Euclidean Jordan algebra and $\mathbb{K} = \mathbb{E}_+$.
For a face $\mathbb{F}$ of $\mathbb{K}$, there exists an idempotent $c$ in the Euclidean Jordan algebra $\mathbb{E}$ such that $\mathbb{F} = \mathbb{E}(c,1)_+$ and $\setspan\mathbb{F} = \mathbb{E}(c,1)$~\cite[Theorem~\mbox{3.1}]{GS2006}.
Therefore, the claim follows from Lemma~\ref{lem:COP_symcone_face}.
\end{proof}

Corollary~\ref{cor:COP_symcone_face_geometric} implies that the facial structure of the copositive cone over a symmetric cone is never simpler than that of the underlying symmetric cone.
For a symmetric cone $\mathbb{K}$, consider faces $\mathbb{F}_1,\mathbb{F}_2$ of $\mathbb{K}$ such that $\mathbb{F}_2 \subseteq \mathbb{F}_1$.
It follows from Corollary~\ref{cor:COP_symcone_face_geometric} that $\{0\} \oplus \COP(\mathbb{F}_1)$, which is isomorphic to $\COP(\mathbb{F}_1)$, is a face of $\COP(\mathbb{K})$.
In addition, since $\mathbb{F}_1$ is also a symmetric cone on its span and $\mathbb{F}_2$ is a face of $\mathbb{F}_1$, $\{0\} \oplus \COP(\mathbb{F}_2)$ is a face of $\COP(\mathbb{F}_1)$.

In general, the facial structure of the copositive cone over a symmetric cone is much more complicated than the underlying symmetric cone.
For example, while the nonnegative orthant is polyhedral, the standard copositive cone is not even facially exposed.
Similarly, although symmetric cones satisfy good properties including facial exposedness\footnote{In fact, symmetric cones satisfy a stronger form of facial exposedness called \emph{orthogonal projectional exposedness}, see \cite[Proposition~\mbox{33}]{Lourenco2021} and also Propositions~9 and \mbox{13} therein.}, copositive cones over symmetric cones are not facially exposed, as shown in the following theorem.

\begin{theorem}\label{thm:main_geometric}
Let $\mathbb{K}$ be a symmetric cone of dimension greater than or equal to $2$ in a finite-dimensional real inner product space.
Then for any $c$ generating an extreme ray of $\mathbb{K}$, $\mathbb{R}_+c\otimes c$ is a non-exposed face of $\COP(\mathbb{K})$.
In particular, $\COP(\mathbb{K})$ is not facially exposed.
\end{theorem}

\begin{proof}
It follows from Corollary~\ref{cor:COP_symcone_face_geometric} that $\{0\} \oplus \COP(\mathbb{R}_+c)$ is a face of $\COP(\mathbb{K})$.
In addition, as a linear transformation, since $c\otimes c$ is a basis of $\mathcal{S}(\mathbb{R}c)$, $\{0\} \oplus \COP(\mathbb{R}_+c)$ is equal to $\mathbb{R}_+ c\otimes c$.
In what follows, we show that the face $\mathbb{R}_+ c\otimes c$ of $\COP(\mathbb{K})$ is not exposed.

We assume that the face $\mathbb{R}_+c\otimes c$ is exposed.
Then there exists $\mathcal{H} \in \CP(\mathbb{K})$ such that
\begin{equation}
\mathbb{R}_+c\otimes c = \COP(\mathbb{K}) \cap \{\mathcal{H}\}^\perp. \label{eq:nonexposed_ray_contradiction_general}
\end{equation}
As $\mathcal{H} \in \CP(\mathbb{K})$, there exist $h_1,\dots,h_m \in \mathbb{K}$ such that $\mathcal{H}$ decomposes into $\sum_{i=1}^m h_i \otimes h_i$, see \eqref{eq:CP}.
We see from \eqref{eq:nonexposed_ray_contradiction_general} that $c\otimes c$ is orthogonal to $\mathcal{H}$, i.e.,
\begin{equation}\label{eq:tensor_trace}
0 = \langle c\otimes c,\mathcal{H}\rangle = \sum_{i=1}^m\langle c\otimes c,h_i \otimes h_i\rangle.
\end{equation}
Fix an arbitrary orthonormal basis $\{v_1,\ldots, v_n\}$ of $\mathbb{V}$.
For each $i = 1,\dots,m$, we have
$c \bullet h_i = \sum _{j=1}^n (c\bullet v_j)(h_i\bullet v_j)$.
By definition, $\langle c\otimes c,h_i \otimes h_i\rangle$ is the trace of the functional composition of $c \otimes c$ with $h_i \otimes h_i$.
Therefore, with our choice of basis, we have
\begin{align}
\langle c\otimes c,h_i \otimes h_i\rangle &=  \sum _{j=1}^n v_j \bullet \{(c\otimes c)((h_i \otimes h_i)(v_j) )\} \nonumber\\
&=  \sum _{j=1}^n (v_{j}\bullet c)  (c \bullet h_i) (h_i \bullet v_j)\nonumber\\
&= (c\bullet h_i)^2.\label{eq:cprod_dot_hprod}
\end{align}
Therefore, \eqref{eq:tensor_trace} and \eqref{eq:cprod_dot_hprod} imply that $c\bullet h_i = 0$ for all $i = 1,\dots,m$.

Since the dimension of $\mathbb{K}$ is greater than or equal to $2$ and that of $\mathbb{R}_+c$ is $1$, $\mathbb{R}_+c$ is a face strictly contained in $\mathbb{K}$.
Recalling that symmetric cones are facially exposed, there exists a supporting hyperplane that exposes the face $\mathbb{R}_+c$, i.e., there exists $d \in \mathbb{K}^*$ such that $\mathbb{R}_+c = \mathbb{K} \cap \{d\}^{\perp}$.
Since $\mathbb{R}_+c$ is strictly contained in $\mathbb{K}$ and $\mathbb{K}$ is self-dual,
we have $d \neq 0$ and $d \in \mathbb{K} \cap \{c\}^\perp$.

Let $\mathcal{A} \coloneqq c\otimes d + d\otimes c \in \mathcal{S}(\setspan\mathbb{K})$.
Note that
\begin{equation}
\|\mathcal{A}\|^2 = 2\|c\|^2\|d\|^2 > 0. \label{eq:A_norm}
\end{equation}
For any $x\in \mathbb{K}$, we have
\begin{equation*}
x\bullet \mathcal{A}(x) = 2(c\bullet x)(d\bullet x) \ge 0,
\end{equation*}
where we use $c,d\in\mathbb{K}$ and the self-duality of $\mathbb{K}$ to derive the inequality.
Therefore, we obtain $\mathcal{A} \in \COP(\mathbb{K})$.
In addition, since
\begin{equation*}
\langle \mathcal{A},\mathcal{H}\rangle = 2\sum_{i=1}^m (c\bullet h_i)(d\bullet h_i) = 0,
\end{equation*}
we see that $\mathcal{A} \in \COP(\mathbb{K}) \cap \{\mathcal{H}\}^\perp$.
Combining it with \eqref{eq:nonexposed_ray_contradiction_general} implies that there exists $\alpha \ge 0$ such that $\mathcal{A} = \alpha c\otimes c$.
Therefore, we have
\begin{equation*}
\|\mathcal{A}\|^2 = \langle c\otimes d + d\otimes c,\alpha c\otimes c\rangle = 2\alpha\|c\|^2(c\bullet d) = 0,
\end{equation*}
which contradicts \eqref{eq:A_norm}.
Thus, $\mathbb{R}_+c\otimes c$ is a non-exposed face of $\COP(\mathbb{K})$.
\end{proof}

Theorem~\ref{thm:main_geometric} proves the main claim of this paper: that, in general,  $\COP(\mathbb{K})$ is not facially exposed.
However, the situation for $\CP(\mathbb{K})$, the dual of $\COP(\mathbb{K})$, is significantly less clear.
Zhang~\cite[Theorem~\mbox{3.4}]{Zhang2018} showed that the cone of completely positive matrices over $\mathbb{R}^n_+$ (i.e., $\CP(\mathbb{R}^n_+)$) is not facially exposed for $n \geq 5$.
Surprisingly, if $\mathbb{K}$ is a single second-order cone, then $\CP(\mathbb{K})$ is facially exposed.
This is because it can be expressed as the intersection of a positive semidefinite cone and a half-space~\cite[page~\mbox{251}]{SZ2003} and the intersection of facially exposed cones is facially exposed.
However, if $\mathbb{K}$ is a cone of positive semidefinite matrices, whether $\CP(\mathbb{K})$  is facially exposed or not seems to be unknown and it would be an interesting question to explore.

\vspace{0.5cm}
\noindent
{\bf Acknowledgments}
We thank the reviewers for their comments, which helped to improve the paper.
The first author is supported by JSPS Grant-in-Aid for JSPS Fellows JP22KJ1327.
The second author is supported by JSPS Grant-in-Aid for Early-Career Scientists JP23K16844 and JSPS Grant-in-Aid for Scientific Research (B) JP21H03398.

\begin{appendices}
\section{On the eleven cases in the proof of Lemma~\ref{lem:COP_symcone_face}}\label{apdx:eleven_cases}
In this appendix, we explain the reason why a quadruple $(i,j,k,l)$ satisfying \eqref{eq:three_conds} must fall into exactly one of the eleven cases listed in the proof of Lemma~\ref{lem:COP_symcone_face}.
The case separation we describe next is based on checking how many among the indices $(i,j,k,l)$ are equal.
First we recall that  $(i,j,k,l)$ satisfies \eqref{eq:three_conds} if and only if one of the following conditions
is satisfied
\begin{equation}\label{case:A}
i \leq j,\ k \leq l,\ i = k,\ j \leq  l,
\end{equation}
or
\begin{equation}\label{case:B}
i \leq  j,\ k \leq l,\ i < k.
\end{equation}

The first case is when all the indices $i,j,k,l$ are equal, which corresponds to Case~\ref{enum:case1}.
We also note  that $(i,j,k,l)$ satisfying \eqref{eq:three_conds} falls into Case~\ref{enum:case1} if and only if $i = l$.
In order to see that,  we assume that $i = l$.
From \eqref{case:A} and \eqref{case:B}, we have $i \leq k$ which implies $l \le k$, by assumption.
Together with $k \le l$, we obtain  $i = k = l$.
In addition, by $ij \preceq kl = ii$, we obtain  $j \leq i$.
Combining this with $i \le j$ implies that $i = j$.
Therefore, $i,j,k,l$ are identical.
In other words, unless $i,j,k,l$ are identical, $i$ is not equal to $l$.

Next, we examine what happens when the quadruple $(i,j,k,l)$ is such that exactly three of $i,j,k,l$ are identical.
They must fall into exactly one of the following four cases:
\begin{enumerate}
\renewcommand{\theenumi}{\Rnum{1}-\arabic{enumi}}
\renewcommand{\labelenumi}{(\theenumi)}
\item $j = k = l \neq i$ (Case~\ref{enum:case5});
\item $i = k = l \neq j$; \label{enum:three_equal_1}
\item $i = j = l \neq k$; \label{enum:three_equal_2}
\item $i = j = k \neq l$ (Case~\ref{enum:case2}).
\end{enumerate}
If the quadruple $(i,j,k,l)$ satisfies \eqref{enum:three_equal_1} or \eqref{enum:three_equal_2}, as shown in the previous paragraph, $i,j,k,l$ must be identical, which contradicts the assumption that exactly three indices are equal.
Therefore, \eqref{enum:three_equal_1} and \eqref{enum:three_equal_2} do not occur.

Our next task is to consider the case where two among the $i,j,k,l$ are identical, and the remaining indices are also identical, but the two pairs are different from each other.
The quadruple $(i,j,k,l)$ satisfying this condition must fall into exactly one of the following three cases:
\begin{enumerate}
\renewcommand{\theenumi}{\Rnum{2}-\arabic{enumi}}
\renewcommand{\labelenumi}{(\theenumi)}
\item $i = j \neq k = l$ (Case~\ref{enum:case3});
\item $i = k \neq j = l$ (Case~\ref{enum:case4});
\item $i = l \neq j = k$. \label{enum:two_two_equal_1}
\end{enumerate}
If the quadruple $(i,j,k,l)$ satisfies \eqref{enum:two_two_equal_1}, $i,j,k,l$ must be identical, which contradicts the assumption.

Next, the quadruple $(i,j,k,l)$ such that exactly two among the $i,j,k,l$ are identical must fall into exactly one of the following six cases: Cases~\ref{enum:case6} through \ref{enum:case10}, and the case where $i = l$ and $i,j,k$ are different from each other.
However, the last case of $i = l$ cannot occur since this implies that all the $i,j,k,l$ are equal.

Finally, the quadruple $(i,j,k,l)$ such that all the indices are different from each other corresponds to Case~\ref{enum:case11}.
\end{appendices}

\bibliographystyle{plainurl}
\bibliography{NL23_2_ref}
\end{document}